\documentclass[a4paper,10pt]{amsart} 
\usepackage{amsmath}
\usepackage{amssymb}
\usepackage{amsthm}
\usepackage{comment}
\usepackage{shuffle}
\usepackage{hyperref}
\usepackage{here}
\newtheorem{theorem}{Theorem}
\newtheorem*{theorem*}{Theorem}

\newtheorem*{lemma*}{Lemma}
\newtheorem{corollary}{Corollary}
\newtheorem*{corollary*}{Corollary}
\newtheorem{proposition}{Proposition}
\newtheorem*{proposition*}{Proposition}

\theoremstyle{definition}
\newtheorem{definition}{Definition}
\newtheorem*{definition*}{Definition}

\newtheorem*{remark*}{Remark}

\newtheorem*{example*}{Example}

\DeclareMathOperator{\Lt}{Lt}

\allowdisplaybreaks[1]

\begin{document}
\title{Multiple $T$-values and iterated log-tangent integrals}
\author{RYOTA UMEZAWA}
\subjclass[2020]{Primary 11M32, Secondary 33E20}
\keywords{Multiple $T$-values, Log-tangent integral, Bernoulli polynomials, Euler polynomials}
\date{}
\maketitle
\begin{abstract}
Multiple $T$-values, a variant of multiple zeta values of level two, were introduced and studied by Kaneko and Tsumura. This paper will introduce iterated log-tangent integrals and discuss their relations with multiple $T$-values. We will use ideas from the author's previous work on multiple zeta values and iterated log-sine integrals to do so.
\end{abstract}
\section{Introduction}
For an index $\mathbf{k}=(k_{1},\dots,k_{n}) \in \mathbb{Z}_{\ge0}^n$, $|\mathbf{k}|:=k_{1}+\dots+k_{n}$ is called the weight of $\mathbf{k}$ and ${\rm dep}(\mathbf{k}) := n$ is called the depth of $\mathbf{k}$. When $\mathbf{k} \in \mathbb{Z}_{\ge1}^n$ and $k_{n} \ge 2$, $\mathbf{k}$ is called admissible. We regard $\emptyset$ as the admissible index of weight $0$ and depth $0$. We define $\mathbf{k}_{-}=(k_{1},\dots,k_{n-1})$ when $n \ge 2$ and $\mathbf{k}_{-}=\emptyset$ when $n=1$. Let $\{k\}^{m}$ denote $m$ repetitions of $k$ and $\mathbf{1}_{n} = (\{1\}^{n})$. For two indices $\mathbf{k} = (k_{1},\dots,k_{n})$ and $\mathbf{l} = (l_{1},\dots,l_{n})$, we define
\begin{align*}
\mathbf{k}\pm\mathbf{l} = (k_{1}\pm l_{1},\dots,k_{n}\pm l_{n}),
\end{align*}
and write $\mathbf{k} \ge \mathbf{l}$ when $k_{1} \ge l_{1}, \dots, k_{n} \ge l_{n}$.

Multiple $T$-values are defined by Kaneko and Tsumura \cite{KT1, KT} as a variant of multiple zeta values of level two as follows.
\begin{definition}[Multiple $T$-values] For an admissible index $\mathbf{k}=(k_{1},\dots,k_{n})$, we define
\begin{align*}
T(\mathbf{k}) = 2^{n}\sum_{\substack{0<m_{1}<\cdots<m_{n}\\m_{j}\equiv j\ {\rm mod}\ 2}}\frac{1}{m^{k_{1}}_{1}\cdots m^{k_{n}}_{n}}.
\end{align*}
We regard $T(\emptyset)$ as $1$. 
\end{definition}
For $k \in \mathbb{Z}_{\ge 0}$, we define
\begin{align*}
\mathcal{T}_{k}^{\shuffle} = \sum_{\substack{|\mathbf{k}| = k\\\mathbf{k}:\text{admissible}}}\mathbb{Q} \cdot T(\mathbf{k})
\end{align*}
and $d_{k}^{T}:=\dim \mathcal{T}_{k}^{\shuffle}$. Then, Kaneko and Tsumura \cite{KT} gave some relation among multiple $T$-values, for example, duality, sum formula, etc., as well as the following conjectural table by numerical experiments.
\begin{table}[H]
  \begin{tabular}{|c|c|c|c|c|c|c|c|c|c|c|c|c|c|c|c|c|c} \hline
     $k$ & 0 & 1& 2 & 3 & 4 & 5 & 6 & 7 & 8 & 9 & 10 & 11 & 12 & $13$ & $14$ & $15$\\ \hline 
    $d^{T}_{k}$ & 1 & 0 & 1 & 1 & 2 & 2 & 4 & 5 & 9 & 10 & 19& 23 & 42 & $49$ & $91$ & $110$\\ \hline 
  \end{tabular}
\end{table}
Multiple $T$-values have recently been studied by many authors \cite{M}, \cite{T} and \cite{XZ} etc.

This paper introduces iterated log-tangent integrals and discusses relations between them and multiple $T$-values.
\begin{definition}[Iterated log-tangent integrals]
For $\sigma \in \mathbb{R}$, $\mathbf{k} = (k_{1}, \dots, k_{n}) \in \mathbb{Z}_{\ge1}^{n}$ and $\mathbf{l} = (l_{1},\dots,l_{n}) \in \mathbb{Z}_{\ge 0}^{n}$ such that $\mathbf{k} - \mathbf{1}_{n} \ge \mathbf{l}$, we define
\begin{align*}
\Lt_{\mathbf{k}}^{\mathbf{l}}(\sigma) &= (-1)^{n}\int_{0}^{\sigma}\int_{0}^{\theta_{n}}\dots\int_{0}^{\theta_{2}}
\prod_{u=1}^{n}\theta_u^{l_u}\left(\log\left|\tan{\frac{\theta_{u}}{2}}\right|\right)^{k_u-1-l_u}\,d\theta_{1}\cdots d\theta_{n-1}d\theta_{n}\\
&= -\int_{0}^{\sigma} \theta^{l_{n}}\left(\log\left|\tan{\frac{\theta}{2}}\right|\right)^{k_n-1-l_n}\Lt_{\mathbf{k}_{-}}^{\mathbf{l}_{-}}(\theta)\,d\theta,
\end{align*}
where $\Lt_{\emptyset}^{\emptyset}(\sigma)$ is regarded as $1$.
\end{definition}
Iterated log-tangent integrals converge absolutely for any $\sigma \in \mathbb{R}$ when $\mathbf{k} - \mathbf{1}_{n} \ge \mathbf{l}$. The single log-tangent integral is known to be related to the Riemann zeta function and others in \cite{EG}, \cite{EG2} and \cite{S}. The present study is a $T$-value analogue of papers on relations between multiple zeta values and iterated log-sine integrals \cite{U1} and \cite{U2}. Many of the results in the present paper are proved by ideas in \cite{U1} and \cite{U2}. The following function introduced in \cite{KT1, KT} plays an important role in that.
\begin{definition}[Multiple $A$-functions]
For $z \in \mathbb{C}$ and $\mathbf{k} = (k_{1}, \dots, k_{n}) \in \mathbb{Z}_{\ge 1}^{n}$, we define
\begin{align*}
A(\mathbf{k};z) = 2^{n}\sum_{\substack{0<m_{1}<\cdots<m_{n}\\m_{j}\equiv j\ {\rm mod}\ 2}}\frac{z^{m_{n}}}{m^{k_{1}}_{1}\cdots m^{k_{n}}_{n}}.
\end{align*}
\end{definition}
Multiple $A$-functions converge absolutely for $|z| < 1$ and can be continued holomorphically to $\mathbb{C}\setminus((-\infty,-1] \cup [1,\infty))$ by the integral representation (see \cite{X}). We regard $A(\emptyset;z)$ as $1$. 
We call $A(\mathbf{k};i)$ multiple $A$-values. Note that $A(\mathbf{k};i) \in \mathbb{R}$ if $n$ is even and $A(\mathbf{k};i) \in i\mathbb{R}$ if $n$ is odd. The multiple $A$-value can be written as
\begin{align*}
A(\mathbf{k};i) = (2i)^{n}\sum_{\substack{0<m_{1}<\cdots<m_{n}\\m_{j}\equiv j\ {\rm mod}\ 2}}\frac{(-1)^{(m_{n}-n)/2}}{m^{k_{1}}_{1}\cdots m^{k_{n}}_{n}},
\end{align*}
and $\tilde{T}(\mathbf{k}) := (-i)^{n}A(\mathbf{k};i)$ are studied in \cite{KT3}. Following multiple Clausen functions and multiple Glaisher functions (see \cite{BS}), we define $A^{e}(\mathbf{k};\sigma)$ and $A^{o}(\mathbf{k};\sigma)$ as follows.
\begin{definition}
We define
\begin{align*}
A^{e}(\mathbf{k};\sigma)&=\begin{cases}
\Re(A(\mathbf{k};e^{i\sigma}))\quad \text{if}\ |\mathbf{k}|:{\rm even},\\
\Im(A(\mathbf{k};e^{i\sigma}))\quad \text{if}\ |\mathbf{k}|:{\rm odd},
\end{cases}\\
A^{o}(\mathbf{k};\sigma)&=\begin{cases}
\Im(A(\mathbf{k};e^{i\sigma}))\quad \text{if}\ |\mathbf{k}|:{\rm even},\\
\Re(A(\mathbf{k};e^{i\sigma}))\quad \text{if}\ |\mathbf{k}|:{\rm odd}.
\end{cases}
\end{align*}
\end{definition}
In addition, for $k \ge 0$, we define 
\begin{align*}
\mathcal{A}^{e}_{k} = \sum_{|\mathbf{k}| = k}\mathbb{Q}\cdot A^{e}(\mathbf{k};\pi/2)\quad\text{and}\quad\mathcal{A}^{o}_{k} = \sum_{|\mathbf{k}| = k}\mathbb{Q}\cdot A^{o}(\mathbf{k};\pi/2).
\end{align*}
Note that $A^{e}(\mathbf{k};\pi/2) = 0$ if $|\mathbf{k}|+n$ is odd and $A^{o}(\mathbf{k};\pi/2) = 0$ if $|\mathbf{k}|+n$ is even. Let $d_{k}^{e} = \dim{\mathcal{A}_{k}^{e}}$ and $d_{k}^{o} = \dim{\mathcal{A}_{k}^{o}}$. Then, Kaneko and Tsumura \cite{KT3} gave the following conjectural table by numerical experiments.
\begin{table}[H]
  \begin{tabular}{|c|c|c|c|c|c|c|c|c|c|c|c|c|} \hline
     $k$ &0&1& 2 & 3 & 4 & 5 & 6 & 7 & 8 & 9 & 10 & 11 \\ \hline 
    $d_{k}^{e}$ 
	& 1 & 1 & 1 & 1 & 3 & 4 & 8 & 10 & 22 & 29 & 59 & 78 \\ \hline 
    $d_{k}^{o}$ 
	& 0 & 0 & 1 & 2 & 3 & 4 & 8 & 12 & 22 & 30 & 59 & 84 \\ \hline 
  \end{tabular}
\end{table}
Note that this table differs from the table in \cite{KT3} when $k$ is odd due to differences in definitions. 

In Section \ref{sec:ILT}, we will describe several properties of iterated log-tangent integrals and $\mathbb{Q}$-vector spaces
\begin{align*}
\mathcal{L}^{e}_{k} &:= \sum_{\substack{|\mathbf{k}| = k,\\ \mathbf{k}-\mathbf{1}_{{\rm dep}(\mathbf{k})} \ge \mathbf{l},\\|\mathbf{k} - \mathbf{1}_{{\rm dep}(\mathbf{k})}- \mathbf{l}|:\,\text{even}}}\mathbb{Q}\cdot\Lt_{\mathbf{k}}^{\mathbf{l}}\left(\frac{\pi}{2}\right),
\\
\mathcal{L}^{o}_{k} &:= \sum_{\substack{|\mathbf{k}| = k,\\ \mathbf{k}-\mathbf{1}_{{\rm dep}(\mathbf{k})} \ge \mathbf{l},\\|\mathbf{k} - \mathbf{1}_{{\rm dep}(\mathbf{k})}- \mathbf{l}|:\,\text{odd}}}\mathbb{Q}\cdot\Lt_{\mathbf{k}}^{\mathbf{l}}\left(\frac{\pi}{2}\right).
\end{align*}
These properties are similar to the properties of iterated log-sine integrals described in \cite{U1}.

In Section \ref{sec:TandLt}, we will prove that every multiple $T$-value $T(\mathbf{k})$ can be written as a $\mathbb{Q}$-linear combination of the real numbers
\begin{align*}
\left\{ \pi^m \Lt_{\mathbf{k}'}^{\mathbf{l}'}\left(\frac{\pi}{2}\right)\mathrel{}\middle|\mathrel{}  \begin{gathered} m+|\mathbf{k}'| = |\mathbf{k}|,\ m \ge 0,\\
\mathbf{k}' - \mathbf{1}_{{\rm dep}(\mathbf{k}')} > \mathbf{l}',\\
|\mathbf{k}'-\mathbf{1}_{{\rm dep}(\mathbf{k}')}-\mathbf{l}'| \equiv |\mathbf{k}|\ {\rm mod}\ 2
\end{gathered} \right\}.
\end{align*}
We also provide a method for obtaining relations among multiple $T$-values by using its representation. We will also prove that when $\mathbf{k}$ is of height one, i.e., $\mathbf{k} = (\{1\}^{n-1},k+1)$, multiple $T$-values can be expressed in terms of single log-tangent integrals involving the Bernoulli polynomials and the Euler polynomials (see Corollary \ref{co:T(2n,k)} and Corollary \ref{co:T(2n+1,k)}).

In Section \ref{sec:AandLt}, we will prove that the iterated log-tangent integral $\Lt_{\mathbf{k}}^{\mathbf{l}}\left(\sigma\right)$ can be written as a $\mathbb{Q}$-linear combination of 
\begin{align*}
\left\{ i^{|\mathbf{l}|+n}(i\sigma)^{m} (\pi i)^{p} T(\mathbf{k}_{1}) A(\mathbf{k}_{2};e^{i \sigma}) \mathrel{}\middle|\mathrel{}  \begin{gathered}m, p \in \mathbb{Z}_{\ge 0},\ 
\mathbf{k}_{1}, \mathbf{k}_{2}: \text{admissible},\\
m+p+|\mathbf{k}_{1}| + |\mathbf{k}_{2}| = |\mathbf{k}|,\\
p + {\rm dep}(\mathbf{k}_{1})+{\rm dep}(\mathbf{k}_{2}) = |\mathbf{k}-\mathbf{1}_{n}-\mathbf{l}|
 \end{gathered}\right\}.
\end{align*}
As a consequence of this, the following theorem is given.
\begin{theorem}\label{th:L=A}
For $k \in \mathbb{Z}_{\ge 0}$, we have 
\begin{align*}
\mathcal{L}^{e}_{k} = \mathcal{A}_{k}^{e}, \quad\text{and}\quad
\mathcal{L}^{o}_{k} = \mathcal{A}_{k}^{o}.
\end{align*}
\end{theorem}
\section{Properties of iterated log-tangent integrals}\label{sec:ILT}
In this section, we describe some properties of iterated log-tangent integrals.
We note that the following properties hold for iterated log-tangent integrals as well as for iterated log-sine integrals.
\begin{proposition}[Shuffle product formula]\label{pr:Ltshuffle}
We have
\begin{align}{\rm Lt}_{k_{1}, \dots, k_{n}}^{(l_{1}, \dots, l_{n})}(\sigma)\cdot {\rm Lt}_{k_{n+1}, \dots, k_{n+n'}}^{(l_{n+1}, \dots, l_{n+n'})}(\sigma) = \sum_{\tau \in S_{n, n'}}{\rm Lt}_{k_{\tau({1})},\dots,k_{\tau({n+n'})}}^{(l_{\tau({1})},\dots,l_{\tau({n+n'})})}(\sigma),\label{eq:shuffle}
\end{align}
where $S_{n,n'}$ is a subset of $ S_{n+n'}$(the symmetric group of degree $n+n'$) and is defined by
\begin{align*}
\{\tau \in S_{n+n'} \mid \tau^{-1}(1)<\dots<\tau^{-1}(n)\ and \ \tau^{-1}(n+1)<\dots<\tau^{-1}(n+n')\}.
\end{align*}
\end{proposition}
\begin{proposition}\label{pr:j}
For a pair of indices $\mathbf{k}=(k_{1},\dots,k_{n})$ and $\mathbf{l}=(l_{1},\dots,l_{n})$ with $n \ge 2$, we define 
\begin{align*}
\mathbf{k}_{j+}&= (k_{1},\dots,k_{j-1}, k_{j}+k_{j+1},k_{j+2},\dots,k_{n})&(j \in \{1,\dots,n-1\}),\\
\mathbf{k}_{j-}&= (k_{1},\dots,k_{j-2}, k_{j-1}+k_{j},k_{j+1},\dots,k_{n})&(j \in \{2,\dots,n\})
\end{align*}and
\begin{align*}
\mathbf{l}_{j+}&= (l_{1},\dots,l_{j-1}, k_{j}+l_{j+1},l_{j+2},\dots,l_{n})&(j \in \{1,\dots,n-1\}),\\
\mathbf{l}_{j-}&= (l_{1},\dots,l_{j-2}, l_{j-1}+k_{j},l_{j+1},\dots,l_{n})&(j \in \{2,\dots,n\}).
\end{align*}
If there exists $j \in \{1,\dots,n\}$ such that $k_j -1 - l_j = 0$, then the following equality holds:
\begin{align}\label{eq:j}
{\rm Lt}_{\mathbf{k}}^{\mathbf{l}}(\sigma)&=\begin{cases}
-\frac{1}{k_{1}}\sigma^{k_{1}}&{\rm if}\ n=1, j=1,\\
-\frac{1}{k_1}{\rm Lt}_{\mathbf{k}_{1+}}^{\mathbf{l}_{1+}}(\sigma)&{\rm if}\ n \ge 2, j=1, \\
\frac{1}{k_j}\left({\rm Lt}_{\mathbf{k}_{j-}}^{\mathbf{l}_{j-}}(\sigma)
-{\rm Lt}_{\mathbf{k}_{j+}}^{\mathbf{l}_{j+}}(\sigma)\right)&{\rm if}\ n \ge 2, 1<j<n,\\
\frac{1}{k_n}\left({\rm Lt}_{\mathbf{k}_{n-}}^{\mathbf{l}_{n-}}(\sigma)-\sigma^{k_{n}}{\rm Lt}_{\mathbf{k}_{-}}^{\mathbf{l}_{-}}(\sigma)\right)&{\rm if}\ n \ge  2, j = n .
\end{cases}
\end{align}
\end{proposition}
Proposition \ref{pr:Ltshuffle} and Proposition \ref{pr:j} correspond to \cite[Proposition 1]{U1} and \cite[Proposition 2]{U1}, respectively.

We write $\mathcal{L}^{e}_{k}$ as the $\mathbb{Q}$-vector space spanned by all iterated log-tangent integrals at $\pi/2$ with $|\mathbf{k}| = k$ and even $|\mathbf{k}  - \mathbf{1}_{{\rm dep}(\mathbf{k})} - \mathbf{l}|$, and $\mathcal{L}^{o}_{k}$ as the same but with odd $|\mathbf{k}  - \mathbf{1}_{{\rm dep}(\mathbf{k})} - \mathbf{l}|$. Namely,
\begin{align*}
\mathcal{L}^{e}_{k} &:= \sum_{\substack{|\mathbf{k}| = k,\\ \mathbf{k}-\mathbf{1}_{{\rm dep}(\mathbf{k})} \ge \mathbf{l},\\|\mathbf{k} - \mathbf{1}_{{\rm dep}(\mathbf{k})}- \mathbf{l}|:\,\text{even}}}\mathbb{Q}\cdot\Lt_{\mathbf{k}}^{\mathbf{l}}\left(\frac{\pi}{2}\right),
\\
\mathcal{L}^{o}_{k} &:= \sum_{\substack{|\mathbf{k}| = k,\\ \mathbf{k}-\mathbf{1}_{{\rm dep}(\mathbf{k})} \ge \mathbf{l},\\|\mathbf{k} - \mathbf{1}_{{\rm dep}(\mathbf{k})}- \mathbf{l}|:\,\text{odd}}}\mathbb{Q}\cdot\Lt_{\mathbf{k}}^{\mathbf{l}}\left(\frac{\pi}{2}\right).
\end{align*}
Note that $\left|\mathbf{k}-\mathbf{1}_{{\rm dep}(\mathbf{k})}-\mathbf{l}\right|$ is the sum of exponents of $\log\left|\tan(\theta_{u}/2)\right|$. When $\mathbf{k} = \mathbf{l} = \emptyset$ i.e. $\Lt_{\emptyset}^{\emptyset}\left(\pi/2\right) = 1$, we consider $|\mathbf{k} - \mathbf{1}_{{\rm dep}(\mathbf{k})}- \mathbf{l}|=0$. Based on \cite[Theorem 8]{U2}, the number of pairs of indices $\mathbf{k}$ and $\mathbf{l}$ such that $|\mathbf{k}| = k$, $\mathbf{k}-\mathbf{1}_{{\rm dep}(\mathbf{k})} \ge \mathbf{l}$, and $|\mathbf{k} - \mathbf{1}_{{\rm dep}(\mathbf{k})}- \mathbf{l}|$ is even, is $1$ when $k=0$ and $(F_{2k}+F_{k})/2$ when $k \ge 1$, where $F_{k}$ represents the $k$-th Fibonacci number. Similarly, the number of pairs of indices $\mathbf{k}$ and $\mathbf{l}$ such that $|\mathbf{k}| = k$, $\mathbf{k}-\mathbf{1}_{{\rm dep}(\mathbf{k})} \ge \mathbf{l}$, and $|\mathbf{k} - \mathbf{1}_{{\rm dep}(\mathbf{k})}- \mathbf{l}|$ is odd, is $0$ when $k=0$ and $(F_{2k}-F_{k})/2$ when $k \ge 1$. 

By Proposition \ref{pr:Ltshuffle} and Proposition \ref{pr:j}, $\mathcal{L}^{e}_{k}$ and $\mathcal{L}^{o}_{k}$ can be rewritten as follows:
\begin{align*}
\mathcal{L}^{e}_{k} &= \sum_{\substack{m+|\mathbf{k}| = k,\ m \ge 0,\\ \mathbf{k}-\mathbf{1}_{{\rm dep}(\mathbf{k})} > \mathbf{l},\\|\mathbf{k} - \mathbf{1}_{{\rm dep}(\mathbf{k})}- \mathbf{l}|:\,\text{even}}}\mathbb{Q}\cdot\pi^{m}\Lt_{\mathbf{k}}^{\mathbf{l}}\left(\frac{\pi}{2}\right),\\
\mathcal{L}^{o}_{k} &= \sum_{\substack{m+|\mathbf{k}| = k,\ m \ge 0,\\ \mathbf{k}-\mathbf{1}_{{\rm dep}(\mathbf{k})} > \mathbf{l}, \\ |\mathbf{k} - \mathbf{1}_{{\rm dep}(\mathbf{k})}- \mathbf{l}|:\,\text{odd}}}\mathbb{Q}\cdot\pi^{m}\Lt_{\mathbf{k}}^{\mathbf{l}}\left(\frac{\pi}{2}\right).
\end{align*}
Based on \cite[Theorem 10]{U2}, the number of combinations of non-negative integers $m$ and indices $\mathbf{k}$ and $\mathbf{l}$ such that $m+|\mathbf{k}| = k$, $\mathbf{k}-\mathbf{1}_{{\rm dep}(\mathbf{k})} > \mathbf{l}$, and $|\mathbf{k} - \mathbf{1}_{{\rm dep}(\mathbf{k})}- \mathbf{l}|$ is even, is $1$ when $k=0, 1$ and $2^{k-2}$ when $k \ge 2$. Similarly, the number of combinations of non-negative integers $m$ and indices $\mathbf{k}$ and $\mathbf{l}$ such that $m+|\mathbf{k}| = k$, $\mathbf{k}-\mathbf{1}_{{\rm dep}(\mathbf{k})} > \mathbf{l}$, and $|\mathbf{k} - \mathbf{1}_{{\rm dep}(\mathbf{k})}- \mathbf{l}|$ is odd, is $0$ when $k=0, 1$ and $2^{k-2}$ when $k \ge 2$.

\section{Relation between multiple $T$-values and iterated log-tangent integrals}\label{sec:TandLt}
In this section, we prove that multiple $T$-values can be written in terms of iterated log-tangent integrals and provide a method for obtaining relations among multiple $T$-values. We also discuss relations between multiple $T$-values of height one and single log-tangent integrals involving the Bernoulli polynomials and the Euler polynomials.
\subsection{Iterated log-tangent integral representation of multiple $T$-values}
We write an admissible index $\mathbf{k}$ in the form
\begin{align*}
\mathbf{k}=(\{1\}^{a_{1}-1},b_{1}+1,\{1\}^{a_{2}-1},b_{2}+1,\dots,\{1\}^{a_{h}-1},b_{h}+1),
\end{align*}
then its dual index $\mathbf{k}^{*}$ is defined by
\[\mathbf{k}^{*}=(\{1\}^{b_{h}-1},a_{h}+1,\dots,\{1\}^{b_{2}-1},a_{2}+1,\{1\}^{b_{1}-1},a_{1}+1).\]
We define $\mathbf{k}^{(0)}=\mathbf{k}$,
\begin{align*}
\mathbf{k}^{(1)}=&
  \begin{cases}
    (k_{1},\dots,k_{n-1},k_{n}-1)&\quad {\rm if}\ k_{n}>1,\\
    (k_{1},\dots,k_{n-1})&\quad {\rm if}\ k_{n}=1,\\
    \emptyset&\quad {\rm if}\ n=1, k_{n}=1,
  \end{cases}
\end{align*}
and recursively $\mathbf{k}^{(m)}=(\mathbf{k}^{(m-1)})^{(1)}$ when $m > 1$. Then multiple $T$-values can be written in terms of multiple $A$-values as follows.
\begin{theorem}\label{th:TtoA} For an admissible index $\mathbf{k}$, we have
\begin{align}
T(\mathbf{k}) = \sum_{m=0}^{|\mathbf{k}|}A(\mathbf{k}^{(m)};i)\overline{A((\mathbf{k}^{*})^{(|\mathbf{k}|-m)};i)}. \label{eq:TtoA}
\end{align}
\end{theorem}
\begin{proof}
This follows from the path-composition formula (see \cite[Theorem 2.6]{KT3}).
We can also prove this theorem by repeating integration by part. This can be seen by starting with 
\begin{align*}
T(\mathbf{k}) &=  \int_{0}^{i}\frac{A(\mathbf{k}^{(1)};t)}{t}\,dt + \int_{i}^{1}\frac{A(\mathbf{k}^{(1)};t)}{t}\,dt\\
& = A(\mathbf{k};i) - \int_{i}^{1}A(\mathbf{k}^{(1)};t) \left(A(1;\frac{1-t}{1+t})\right)'\,dt
\end{align*}
and using
\begin{align*}
\frac{d}{dt}A(\mathbf{k};t) = \begin{cases}\frac{2}{1-t^2}A(\mathbf{k}^{(1)};t)&(k_{n} = 1), \\ \frac{1}{t}A(\mathbf{k}^{(1)};t)&(k_{n} > 1)\end{cases}
\end{align*}
and
\begin{align*}
\frac{d}{dt}A(\mathbf{k};\frac{1-t}{1+t}) = \begin{cases}-\frac{1}{t}A(\mathbf{k}^{(1)};\frac{1-t}{1+t})&(k_{n} = 1), \\ -\frac{2}{1-t^2}A(\mathbf{k}^{(1)};\frac{1-t}{1+t})&(k_{n} > 1).\end{cases}
\end{align*}
\end{proof}
\begin{corollary}\label{co:TtoA}
Every multiple $T$-value of weight $k$ can be written as a $\mathbb{Q}$-linear combination of multiple $A$-values of weight $k$.
\end{corollary}
This corollary follows from the fact that $\overline{A((\mathbf{k}^{*})^{(|\mathbf{k}|-m)};i)}$ is equal to $A((\mathbf{k}^{*})^{(|\mathbf{k}|-m)};i)$ or $-A((\mathbf{k}^{*})^{(|\mathbf{k}|-m)};i)$ and the shuffle product formula for multiple $A$-functions.
\begin{corollary}
Every A-value $A(\mathbf{k};i)$ with odd weight and odd depth can be written as a sum of the product of two multiple $A$-values of weights less than $|\mathbf{k}|$, with integer coefficients.
\end{corollary}
\begin{proof}
If $\mathbf{k}$ is an admissible index with odd weight and odd depth, then\\ $A(\mathbf{k}^{(m)};i)\overline{A((\mathbf{k}^{*})^{(|\mathbf{k}|-m)};i)}$ is a purely imaginary number when $m$ is even, and a real number when $m$ is odd.
Therefore, by taking the imaginary part of the formula (\ref{eq:TtoA}), we obtain
\begin{align*}
A(\mathbf{k};i) = -\sum_{m=1}^{(|\mathbf{k}|-1)/2}A(\mathbf{k}^{(2m)};i)\overline{A((\mathbf{k}^{*})^{(|\mathbf{k}|-2m)};i)}.
\end{align*}
Note that both $|\mathbf{k}^{(2m)}|$ and $|(\mathbf{k}^{*})^{(|\mathbf{k}|-2m)}|$ are less than $|\mathbf{k}|$ for $m = 1, \dots, (|\mathbf{k}|-1)/2$.

In order to prove the assertion for $\mathbf{k} = (k_{1}, \dots, k_{n-m}, \{1\}^{m})$, we use the algebraic setup introduced by Hoffman \cite{H}. We define the noncommutative polynomial ring $\mathfrak{H} = \mathbb{Q}\langle e_{0}, e_{1}\rangle$ and its subspace $\mathfrak{H}^{0} = \mathbb{Q}+e_{1}\mathfrak{H}e_{0}$ and $\mathfrak{H}^{1} = \mathbb{Q}+e_{1}\mathfrak{H}$. The shuffle product $\shuffle$ in $\mathfrak{H}$ is recursively defined by
\begin{align*}
w\shuffle1&=w,\ \,\,1\shuffle w=w,\\
u_{1}w_{1}\shuffle u_{2}w_{2}&=u_{1}(w_{1}\shuffle u_{2}w_{2})+u_{2}(u_{1}w_{1}\shuffle w_{2})
\end{align*}
$(u_{1},u_{2} \in \{e_{0},e_{1}\}, w, w_{1},w_{2}:\text{words})$ with $\mathbb{Q}$-bilinearity. We also use $T$ and $A$ as the $\mathbb{Q}$-linear maps $T\colon \mathfrak{H}^{0} \to \mathbb{R}$ and $A(\,\cdot\,;z )\colon \mathfrak{H}^{1}\to \mathbb{C}$ by
\begin{align*}
T(e_{1} e_{0}^{k_{1}-1}\cdots e_{1}e_{0}^{k_{n}-1}) & = T(k_{1},\dots,k_{n}),&&& T(1_{\mathfrak{H}}) &= T(\emptyset) = 1,
\\
A(e_{1} e_{0}^{k_{1}-1}\cdots e_{1}e_{0}^{k_{n}-1};z) &= A\left({k_{1},\dots,k_{n}};z\right),&&& A(1_{\mathfrak{H}};z) &= A(\emptyset;z)=1,
\end{align*}
respectively. Note that $A(w_{1};z)A(w_{2};z)=A(w_{1}\shuffle w_{2};z)$ for any $w_{1}, w_{2} \in \mathfrak{H}^{1}$.
By the explicit regularization formula (see \cite[Corollary 5 and (5.2)]{IKZ}):
\begin{align*}
we_{1}^{m} = \sum_{j=0}^{m}(-1)^{m-j}\left(\left(w' \shuffle e_{1}^{m-j}\right)e_{0}\right) \shuffle e_{1}^{j}\quad (w = w'e_{0} \in \mathfrak{H}^{0}),
\end{align*}
we have
\begin{align*}
&A(k_{1}, \dots, k_{n-m}, \{1\}^{m};i) \\
&= \sum_{j=0}^{m}(-1)^{m-j}A\left(\left(w' \shuffle e_{1}^{m-j}\right)e_{0};i\right)A(\{1\}^{j};i)\\
& = (-1)^{m}A\left(\left(w' \shuffle e_{1}^{m}\right)e_{0};i\right) + \sum_{j=1}^{m}(-1)^{m-j}A\left(\left(w' \shuffle e_{1}^{m-j}\right)e_{0};i\right)A(\{1\}^{j};i),
\end{align*}
where $w' = e_{1}e_{0}^{k_{1}-1} \cdots e_{1}e_{0}^{k_{n-m}-1}$. The assertion follows since the term\\
$A\left(\left(w' \shuffle e_{1}^{m}\right)e_{0};i\right)$ is a sum of multiple $A$-values of admissible indices with odd weights and odd depths, with integer coefficients.
\end{proof}
Multiple $A$-values can be written in terms of iterated log-tangent integrals as follows.
\begin{theorem}\label{th:AtoLt} For $\mathbf{k} \in \mathbb{Z}_{\ge 1}^{n}$ and $\sigma \in [0,\pi)$, we have
\begin{align*}
&A(\mathbf{k};i\tan(\sigma/2))\\
&=i^n\int_{0<\theta_{1}<\cdots<\theta_{n}<\theta_{n+1}=\sigma}
\prod_{u=1}^{n}\frac{\left(\log\tan(\theta_{u+1}/{2})-\log\tan(\theta_{u}/{2})\right)^{k_{u}-1}}{(k_{u}-1)!}\,d\theta_{u}.
\end{align*}
In particular, 
\begin{align}
&A(\mathbf{k};i) \label{eq:AKi}\\
&=i^n\int_{0<\theta_{1}<\cdots<\theta_{n}<\theta_{n+1}=\pi/2}
\prod_{u=1}^{n}\frac{\left(\log\tan(\theta_{u+1}/2)-\log\tan(\theta_{u}/2)\right)^{k_{u}-1}}{(k_{u}-1)!}\,d\theta_{u}. \nonumber
\end{align}
\end{theorem}
\begin{proof}
For $z \in \mathbb{C}\setminus((-\infty,-1]\cup[1,\infty))$, the following integral representation holds:
\begin{align*}
A(\mathbf{k};z)&=\int_{0}^{z}\frac{\log^{k_{n}-1}(z/z_{n})}{(k_{n}-1)!}\frac{2dz_{n}}{1-z_{n}^2}\int_{0}^{z_{n}}\frac{\log^{k_{n-1}-1}(z_{n}/z_{n-1})}{(k_{n-1}-1)!}\frac{2dz_{n-1}}{1-z_{n-1}^2}\\
&\quad\cdots \int_{0}^{z_{2}}\frac{\log^{k_{1}-1}(z_{2}/z_{1})}{(k_{1}-1)!}\frac{2dz_{1}}{1-z_{1}^2}.
\end{align*}
By putting $z= i\tan(\sigma/2)$, $z_{u}=i\tan(\theta_{u}/2)$, we obtain the desired formula. The formula (\ref{eq:AKi}) is given by putting $\sigma=\pi/2$.
\end{proof}
The formula (\ref{eq:AKi}) implies
\begin{align} \label{eq:AinL}
A(\mathbf{k};i) \in \begin{cases} 
\mathcal{L}_{|\mathbf{k}|}^{e}&{\rm if}\quad |\mathbf{k}|:{\rm even},\, {\rm dep}(\mathbf{k}):{\rm even},\\ 
\mathcal{L}_{|\mathbf{k}|}^{o}&{\rm if}\quad |\mathbf{k}|:{\rm odd},\, {\rm dep}(\mathbf{k}):{\rm even},\\
i\mathcal{L}_{|\mathbf{k}|}^{e}&{\rm if}\quad |\mathbf{k}|:{\rm odd},\, {\rm dep}(\mathbf{k}):{\rm odd},\\
i\mathcal{L}_{|\mathbf{k}|}^{o}&{\rm if}\quad |\mathbf{k}|:{\rm even},\, {\rm dep}(\mathbf{k}):{\rm odd}.
\end{cases}
\end{align}
Therefore, the formula (\ref{eq:TtoA}) implies
\begin{align}\label{eq:TinLt}
T(\mathbf{k}) \in \begin{cases} 
\mathcal{L}_{|\mathbf{k}|}^{e} + i\mathcal{L}_{|\mathbf{k}|}^{o}& {\rm if}\quad |\mathbf{k}|:{\rm even},\\
\mathcal{L}_{|\mathbf{k}|}^{o} + i\mathcal{L}_{|\mathbf{k}|}^{e}& {\rm if}\quad |\mathbf{k}|:{\rm odd},
\end{cases}
\end{align}
namely, every multiple $T$-value $T(\mathbf{k})$ can be written as a $\mathbb{Q}$-linear combination of the real numbers
\begin{align*}
\left\{ \pi^m \Lt_{\mathbf{k}'}^{\mathbf{l}'}\left(\frac{\pi}{2}\right)\mathrel{}\middle|\mathrel{}  \begin{gathered} m+|\mathbf{k}'| = |\mathbf{k}|,\ m \ge 0,\\
\mathbf{k}' - \mathbf{1}_{{\rm dep}(\mathbf{k}')} > \mathbf{l}',\\
|\mathbf{k}'-\mathbf{1}_{{\rm dep}(\mathbf{k}')}-\mathbf{l}'| \equiv |\mathbf{k}|\ {\rm mod}\ 2
\end{gathered} \right\}.
\end{align*}
\subsection{Relations of multiple $T$-values}
We calculate examples of the formula (\ref{eq:TtoA}) and the formula (\ref{eq:AKi}) and obtain relations among multiple $T$-values. First, note that when $\mathbf{k}$ is of height one, i.e., $\mathbf{k} = (\{1\}^{n-1},k+1)$, by applying the formula (\ref{eq:AKi}) to the formula (\ref{eq:TtoA}), noting that $A(\{1\}^{n};i) = (i \pi/2)^n/n!$ and 
\begin{align*}
A(\{1\}^{n-1},k+1;i)&=i^n\int_{0}^{\frac{\pi}{2}}\frac{\theta^{n-1}}{(n-1)!}
\frac{\left(-\log\tan\frac{\theta}{2}\right)^{k}}{k!}\,d\theta_{u},
\end{align*}we obtain the following simple formula:
\begin{align}
&T(\{1\}^{n-1},k+1) \label{eq:Tnk}\\
&=\frac{i^{n}}{(n-1)!k!}\int_{0}^{\pi/2}\theta^{n-1}\left(-\frac{\pi}{2}i-\log{\tan{\frac{\theta}{2}}}\right)^{k}\,d\theta \nonumber \\
&\quad+\frac{(-i)^{k}}{n!(k-1)!}\int_{0}^{\pi/2}\theta^{k-1}\left(\frac{\pi}{2}i-\log{\tan{\frac{\theta}{2}}}\right)^{n}\,d\theta-\frac{(i\pi/2)^{n}}{n!}\frac{(-i\pi/2)^{k}}{k!}. \nonumber
\end{align}
This formula is an analogue of the following formula:
\begin{align*}
&\zeta(\{1\}^{n-1},k+1) \\
&=\frac{i^{n}}{(n-1)!k!}\int_{0}^{\pi/3}\theta^{n-1}\left(i\frac{\theta-\pi}{2}-\log{\left(2\sin{\frac{\theta}{2}}\right)}\right)^{k}\,d\theta\nonumber \\
&\quad+\frac{(-i)^{k}}{n!(k-1)!}\int_{0}^{\pi/3}\theta^{k-1}\left(-i\frac{\theta-\pi}{2}-\log{\left(2\sin{\frac{\theta}{2}}\right)}\right)^{n}\,d\theta - \frac{\left(i\pi/3\right)^{n}}{n!}\frac{\left(-i\pi/3\right)^{k}}{k!}, \nonumber
\end{align*}
which can be found \cite{BBK} and \cite{U1}. 

Next, we explain how to find relations among multiple $T$-values by using the formula (\ref{eq:TtoA}) and the formula (\ref{eq:AKi}). The following equations are obtained by applying the formula (\ref{eq:TtoA}) and the formula (\ref{eq:AKi}) to multiple $T$-values  of weights up to $4$:
\begin{align*}
&T(2) = \frac{\pi^2}{4},\\
&T(3) = -\frac{\pi^3}{16} i + \frac{\pi}{2} \, {\rm Lt}_{2}^{(0)}\left(\frac{\pi}{2}\right) - {\rm Lt}_{3}^{(1)}\left(\frac{\pi}{2}\right) - \frac{1}{2} i \, {\rm Lt}_{3}^{(0)}\left(\frac{\pi}{2}\right),\\
&T(1,2) = \frac{\pi^3}{16} i + \frac{\pi}{2} \, {\rm Lt}_{2}^{(0)}\left(\frac{\pi}{2}\right) - {\rm Lt}_{3}^{(1)}\left(\frac{\pi}{2}\right) + \frac{1}{2} i \, {\rm Lt}_{3}^{(0)}\left(\frac{\pi}{2}\right),\\
&T(4) = -\frac{\pi^4}{96} - \frac{\pi^2}{8} i \, {\rm Lt}_{2}^{(0)}\left(\frac{\pi}{2}\right) - \frac{\pi}{4} \, {\rm Lt}_{3}^{(0)}\left(\frac{\pi}{2}\right) + \frac{1}{2} i \, {\rm Lt}_{4}^{(2)}\left(\frac{\pi}{2}\right) + \frac{1}{6} i \, {\rm Lt}_{4}^{(0)}\left(\frac{\pi}{2}\right),\\
&T(1,3) = \frac{\pi^4}{64} + {\rm Lt}_{4}^{(1)}\left(\frac{\pi}{2}\right),\\
&T(2,2) = -2 \, {\rm Lt}_{4}^{(1)}\left(\frac{\pi}{2}\right),\\
&T(1,1,2) = -\frac{\pi^4}{96} + \frac{\pi^2}{8} i \, {\rm Lt}_{2}^{(0)}\left(\frac{\pi}{2}\right) - \frac{\pi}{4} \, {\rm Lt}_{3}^{(0)}\left(\frac{\pi}{2}\right) - \frac{1}{2} i \, {\rm Lt}_{4}^{(2)}\left(\frac{\pi}{2}\right) - \frac{1}{6} i \, {\rm Lt}_{4}^{(0)}\left(\frac{\pi}{2}\right).
\end{align*}
The following equations are obtained by comparing the real and imaginary parts of the above equations:
\begin{align}
&T(3) = T(1,2) = \int_{0}^{\pi/2}\left(\theta - \frac{\pi}{2}\right)\log{\tan{\frac{\theta}{2}}}\,d\theta, \nonumber\\
&{\rm Lt}_{3}^{(0)}\left(\frac{\pi}{2}\right) = -\frac{1}{8} \, \pi^{3}, \label{eq:Lt30} \\
&\frac{1}{8} \, \pi^{2} {\rm Lt}_{2}^{(0)}\left(\frac{\pi}{2}\right)  = \frac{1}{2} \, {\rm Lt}_{4}^{(2)}\left(\frac{\pi}{2}\right) + \frac{1}{6} \, {\rm Lt}_{4}^{(0)}\left(\frac{\pi}{2}\right). \label{eq:Lt40}
\end{align}
By the equation (\ref{eq:Lt30}), all multiple $T$-values of weight $4$ are represented as follows:
\begin{align*}
T(4) = T(1,1,2) = \frac{\pi^4}{48}, \quad T(1,3) = \frac{1}{64} \, \pi^{4} + {\rm Lt}_{4}^{(1)}\left(\frac{\pi}{2}\right), \quad T(2,2) = -2 \, {\rm Lt}_{4}^{(1)}\left(\frac{\pi}{2}\right).
\end{align*}
Therefore, we can see that $\dim \mathcal{T}^{\shuffle}_{4} \le 2$.
The following equations are obtained by applying the formula (\ref{eq:TtoA}) and the formula (\ref{eq:AKi}) to all multiple $T$-values of weight $5$ and taking their real parts:
\begin{align}
&T(5) = T(1,1,1,2) = -\frac{\pi^3}{48} \, {\rm Lt}_{2}^{(0)}\left(\frac{\pi}{2}\right) + \frac{\pi}{12} \, {\rm Lt}_{4}^{(0)}\left(\frac{\pi}{2}\right) + \frac{1}{6} \, {\rm Lt}_{5}^{(3)}\left(\frac{\pi}{2}\right), \label{eq:T5}\\
&T(1,4) = T(1,1,3) = \frac{\pi^2}{8} \, {\rm Lt}_{3}^{(1)}\left(\frac{\pi}{2}\right) - \frac{\pi}{4} \, {\rm Lt}_{4}^{(2)}\left(\frac{\pi}{2}\right) - \frac{1}{6} \, {\rm Lt}_{5}^{(1)}\left(\frac{\pi}{2}\right), \nonumber\\
&T(2,3) = T(1,2,2) =\frac{\pi^3}{16} \, {\rm Lt}_{2}^{(0)}\left(\frac{\pi}{2}\right) - \frac{\pi^2}{2} \, {\rm Lt}_{3}^{(1)}\left(\frac{\pi}{2}\right) + \frac{3\pi}{4} \, {\rm Lt}_{4}^{(2)}\left(\frac{\pi}{2}\right) \nonumber\\
& \hspace{32mm}+ \frac{1}{2} \, {\rm Lt}_{5}^{(1)}\left(\frac{\pi}{2}\right) + \frac{1}{2} \, {\rm Lt}_{2 , 3}^{(0 , 0)}\left(\frac{\pi}{2}\right), \nonumber\\
&T(3,2) = T(2,1,2) = -\frac{\pi^3}{16} \, {\rm Lt}_{2}^{(0)}\left(\frac{\pi}{2}\right) + \frac{\pi^2}{2} \, {\rm Lt}_{3}^{(1)}\left(\frac{\pi}{2}\right) - \frac{3\pi}{4} \, {\rm Lt}_{4}^{(2)}\left(\frac{\pi}{2}\right) \nonumber\\
& \hspace{32mm} - \frac{1}{2} \, {\rm Lt}_{5}^{(1)}\left(\frac{\pi}{2}\right) - \frac{3}{2} \, {\rm Lt}_{2 , 3}^{(0 , 0)}\left(\frac{\pi}{2}\right). \nonumber
\end{align}
In addition, the following equations are obtained by applying the formula (\ref{eq:TtoA}) and the formula (\ref{eq:AKi}) to $T(1,5)$ and $T(2,4)$ and taking their imaginary parts, respectively:
\begin{align*}
0&=-\frac{\pi^3}{48} \, {\rm Lt}_{3}^{(1)}\left(\frac{\pi}{2}\right) + \frac{\pi}{12} \, {\rm Lt}_{5}^{(3)}\left(\frac{\pi}{2}\right) + \frac{\pi}{12} \, {\rm Lt}_{5}^{(1)}\left(\frac{\pi}{2}\right),\\
0&=-\frac{\pi^4}{96} \, {\rm Lt}_{2}^{(0)}\left(\frac{\pi}{2}\right) + \frac{\pi^3}{24} \, {\rm Lt}_{3}^{(1)}\left(\frac{\pi}{2}\right) + \frac{\pi^2}{8} \, {\rm Lt}_{4}^{(2)}\left(\frac{\pi}{2}\right) \\
&\quad - \frac{\pi}{3} \, {\rm Lt}_{5}^{(3)}\left(\frac{\pi}{2}\right) - \frac{\pi}{4} \, {\rm Lt}_{5}^{(1)}\left(\frac{\pi}{2}\right) - \frac{\pi}{4} \, {\rm Lt}_{2 , 3}^{(0 , 0)}\left(\frac{\pi}{2}\right).
\end{align*}
Applying the above equations divided by $\pi$ and the equation (\ref{eq:Lt40}) to the equations (\ref{eq:T5}), we obtain the following equations:
\begin{align*}
T(5) & = T(1,1,1,2) \\
&= -\frac{1}{6}\int_{0}^{\pi/2}\left(\theta - \frac{\pi}{2}\right)^3 \log{\tan{\frac{\theta}{2}}}\,d\theta +\frac{\pi^2}{8}\int_{0}^{\pi/2}\left(\theta - \frac{\pi}{2}\right) \log{\tan{\frac{\theta}{2}}}\,d\theta,\\
T(1,4) &= T(1,1,3) \\
&= -\frac{1}{6}\int_{0}^{\pi/2}\left(\theta - \frac{\pi}{2}\right)^3 \log{\tan{\frac{\theta}{2}}}\,d\theta + \frac{\pi^2}{24}\int_{0}^{\pi/2}\left(\theta - \frac{\pi}{2}\right) \log{\tan{\frac{\theta}{2}}}\,d\theta,\\
T(2,3) &= T(1,2,2)\\
&= \frac{2}{3}\int_{0}^{\pi/2}\left(\theta - \frac{\pi}{2}\right)^3 \log{\tan{\frac{\theta}{2}}}\,d\theta - \frac{\pi^2}{12}\int_{0}^{\pi/2}\left(\theta - \frac{\pi}{2}\right) \log{\tan{\frac{\theta}{2}}}\,d\theta,\\
T(3,2) &= T(2,1,2)\\
&= -\int_{0}^{\pi/2}\left(\theta - \frac{\pi}{2}\right)^3 \log{\tan{\frac{\theta}{2}}}\,d\theta + \frac{\pi^2}{4}\int_{0}^{\pi/2}\left(\theta - \frac{\pi}{2}\right) \log{\tan{\frac{\theta}{2}}}\,d\theta.
\end{align*}
Therefore, we can see that $\dim \mathcal{T}^{\shuffle}_{5} \le 2$. 

Let $l_{k}$ be the upper bound of the dimension of $\mathcal{T}^{\shuffle}_{k}$ obtained by this method. Then the following table is obtained by using a computer. The algorithm used is similar to the one described in \cite[Section 5]{U1}.
\begin{table}[H]
  \begin{tabular}{|c|c|c|c|c|c|c|c|c|c|c|c|} \hline
     $k$ & 2 & 3 & 4 & 5 & 6 & 7 & 8 & 9 & 10 & 11 & 12  \\ \hline 
    $d^{T}_{k}$ & 1 & 1 & 2 & 2 & 4 & 5 & 9 & 10 & 19& 23 & 42  \\ \hline 
    $l_{k}$ & 1 & 1 & 2 &2&5&6&16&14&49& 44 & 169 \\ \hline 
    $2^{k-2}$ & 1 & 2 & 4 &8&16&32&64&128&256& 512& 1024  \\ \hline
  \end{tabular}
\end{table}
On weight $6$, the last equation that cannot be obtained by the method is
\begin{align*}
3T(2,4) + 2T(3,3) = -\frac{15}{7}T(6) + \frac{120}{7}T(1,5) + \frac{60}{7}T(2,4)  + \frac{20}{7}T(3,3).
\end{align*}
This equation is written in \cite{KT} and no proof is given in that paper.
\subsection{Log-tangent integrals involving the Bernoulli polynomials and the Euler polynomials}
The equation (\ref{eq:Tnk}), obtained from the formula (\ref{eq:TtoA}) and the formula (\ref{eq:AKi}), can be rewritten into an expression involving the Bernoulli polynomials and the Euler polynomials. The first is the following theorem on the Euler polynomials.
\begin{theorem}\label{th:genT1}
For $|X| < 1$, we have
\begin{align}
&2\sum_{b=2}^{\infty}\sum_{k=1}^{b-1}\frac{(1-2^{b-k})B_{b-k}}{(b-k)!}\left(- \pi i\right)^{b-k-1}T(\{1\}^{a-1},k+1)X^b \label{eq:genT1}\\
&=\frac{i^{a}}{(a-1)!}X\sec\left(\frac{\pi X}{2}\right)\int_{0}^{\pi/2}\theta^{a-1}e^{-X\log{\tan(\theta/2)}}\,d\theta \nonumber\\
&\quad-\frac{i}{a!}\int_{0}^{\pi/2}\frac{2X^{2}e^{-i\theta X}}{e^{-\pi i X} + 1}\left(\frac{\pi}{2}i-\log{\tan{\frac{\theta}{2}}}\right)^{a}\,d\theta \nonumber\\
&\quad-\frac{(i\pi/2)^{a}}{a!}X\sec\left(\frac{\pi X}{2}\right), \nonumber
\end{align}
where $B_{k}$ denotes the $k$-th Bernoulli number.
\end{theorem}
Note that Theorem \ref{th:genT1} will be proved only by the equation (\ref{eq:Tnk}) and calculation of formal power series. Since $B_{1} = -1/2$ and $B_{2k+1} = 0$ $(k \ge 1)$,  the real part of the coefficient of $X^{b}$ $(b \in \mathbb{Z}_{\ge 2})$ of the left-hand side of (\ref{eq:genT1}) is equal to $T(\{1\}^{a-1}, b)$. Therefore, putting $a=2n+1$ $(n \in \mathbb{Z}_{\ge 0})$ and taking the real part of the coefficient of $X^{b}$, we obtain
\begin{align*}
T(\{1\}^{2n}, b) = \Re\left(\frac{-i}{(2n+1)!}\frac{(-\pi i)^{b-2}}{(b-2)!}\int_{0}^{\pi/2}E_{b-2}\left(\frac{\theta}{\pi}\right)\left(\frac{\pi}{2}i-\log{\tan{\frac{\theta}{2}}}\right)^{2n+1}\,d\theta\right),
\end{align*}
where $E_{k}(t)$ denotes the $k$-th Euler polynomial defined by the generating function
\begin{align*}
\frac{2e^{tX}}{e^{X}+1} = \sum_{k=0}^{\infty} E_{k}(t)\frac{X^{k}}{k!}.
\end{align*}
Therefore, for $k \in \mathbb{Z}_{\ge 1}$, we obtain the following corollary.
\begin{corollary}\label{co:T(2n,k)} For $n \in \mathbb{Z}_{\ge 0}$ and $k \in \mathbb{Z}_{\ge 1}$, we have
\begin{align*}
&T(\{1\}^{2n}, 2k) \\
&=\frac{\pi^{2k-2}}{(2k-2)!}\sum_{m=0}^{n}\frac{(-1)^{k+m-1}(\pi/2)^{2m+1}}{(2m+1)!(2n-2m)!}\int_{0}^{\pi/2}E_{2k-2}\left(\frac{\theta}{\pi}\right)\left(\log{\tan{\frac{\theta}{2}}}\right)^{2n-2m}\,d\theta
\end{align*}
and
\begin{align*}
&T(\{1\}^{2n}, 2k+1) \\
&=\frac{\pi^{2k-1}}{(2k-1)!}\sum_{m=0}^{n}\frac{(-1)^{k+m-1}(\pi/2)^{2m}}{(2m)!(2n+1-2m)!}\int_{0}^{\pi/2}E_{2k-1}\left(\frac{\theta}{\pi}\right)\left(\log{\tan{\frac{\theta}{2}}}\right)^{2n+1-2m}\,d\theta.
\end{align*}
\end{corollary}
In particular, when $n=0$, we have
\begin{align}
T(2k) 
&=(-1)^{k-1}\frac{\pi^{2k-1}}{2(2k-2)!}\int_{0}^{\pi/2}E_{2k-2}\left(\frac{\theta}{\pi}\right)\,d\theta 
=(-1)^{k}\frac{(1-2^{2k})B_{2k}}{(2k)!}\pi^{2k} \label{eq:T(2k)}
\end{align}
and
\begin{align}
&T(2k+1)=(-1)^{k-1}\frac{\pi^{2k-1}}{(2k-1)!}\int_{0}^{\pi/2}E_{2k-1}\left(\frac{\theta}{\pi}\right)\log{\tan{\frac{\theta}{2}}}\,d\theta \label{eq:T(2k+1)}.
\end{align}
Noting $T(k) = 2(1-1/2^{k})\zeta(k)$, the formula (\ref{eq:T(2k)}) is the same as the famous formula
\begin{align*}
\zeta(2k) = \frac{(-1)^{k+1}B_{2k}}{2 (2k)!}(2\pi)^{2k}.
\end{align*}
The formula (\ref{eq:T(2k+1)}) is the same as
\begin{align*}
\int_{0}^{\pi/4}E_{2n-1}\left(\frac{2}{\pi}x\right)\log(\tan{x})\,dx = \frac{(-1)^{n-1}(2n-1)!}{\pi^{2n-1}}\left(1-\frac{1}{2^{2n+1}}\right)\zeta(2n+1)
\end{align*}
which was proved in \cite{EG}.
\begin{proof}[Proof of Theorem \ref{th:genT1}]
By the formula (\ref{eq:Tnk}), we have
\begin{align*}
&\sum_{k=1}^{b-1}\frac{(1-2^{b-k})B_{b-k}}{(b-k)!}\left(-\pi i\right)^{b-k-1}T(\{1\}^{a-1},k+1)\\
&=\frac{i^{a}}{(a-1)!}\sum_{k=0}^{b-1}\frac{(1-2^{b-k})B_{b-k}}{k!(b-k)!}\left(-\pi i\right)^{b-k-1}\int_{0}^{\pi/2}\theta^{a-1}\left(-\frac{\pi}{2}i-\log{\tan{\frac{\theta}{2}}}\right)^{k}\,d\theta \\
&\quad+\frac{1}{a!}\sum_{k=1}^{b-1}\frac{(-i)^{k}(1-2^{b-k})B_{b-k}}{(k-1)!(b-k)!}\left(-\pi i\right)^{b-k-1}\int_{0}^{\pi/2}\theta^{k-1}\left(\frac{\pi}{2}i-\log{\tan{\frac{\theta}{2}}}\right)^{a}\,d\theta\\
&\quad-\frac{(i\pi/2)^{a}}{a!}\sum_{k=0}^{b-1}\frac{(1-2^{b-k})B_{b-k}}{(b-k)!}\left(-\pi i\right)^{b-k-1}\frac{(-i\pi/2)^{k}}{k!}.
\end{align*}
Since
\begin{align*}
\sum_{k=1}^{\infty}\frac{(1-2^{k})B_{k}}{k!}\left(-\pi iX\right)^{k} = \frac{\pi  X}{2}\tan\left(\frac{\pi X}{2}\right) -\frac{\pi X}{2}i = \frac{-\pi i X}{e^{- \pi i X} + 1},
\end{align*}
it follows that
\begin{align*}
&\sum_{b=2}^{\infty}\sum_{k=1}^{b-1}\frac{(1-2^{b-k})B_{b-k}}{(b-k)!}\left(-\pi i\right)^{b-k-1}T(\{1\}^{a-1},k+1)X^b\\
&=\frac{i^{a}}{(a-1)!}\frac{ X}{e^{-\pi i X} + 1}\int_{0}^{\pi/2}\theta^{a-1}e^{\left(-i\pi/2-\log{\tan{(\theta/2)}}\right)X}\,d\theta\\
&\quad-\frac{i}{a!}\frac{ X}{e^{- \pi i X} + 1}\int_{0}^{\pi/2}Xe^{-i\theta X}\left(\frac{\pi}{2}i-\log{\tan{\frac{\theta}{2}}}\right)^{a}\,d\theta\\
&\quad-\frac{(i\pi/2)^{a}}{a!}\frac{ X}{e^{- \pi i X} + 1}e^{(-i\pi/2) X}.
\end{align*}
Therefore, the desired formula is obtained.
\end{proof}
The second is the following theorem on the Bernoulli polynomials.
\begin{theorem}\label{th:genT2}
For $|X| < 1$, we have
\begin{align}
&-2\sum_{b=1}^{\infty}\sum_{k=1}^{b}\frac{B_{b-k}}{(b-k)!}\left(- \pi i \right)^{b-k-1}T(\{1\}^{a-1},k+1)X^b \label{eq:genT2}\\
&=\frac{i^{a-1}}{(a-1)!}X\csc\left(\frac{\pi X}{2}\right)\int_{0}^{\pi/2}\theta^{a-1}e^{-X\log{\tan(\theta/2)}}\,d\theta \nonumber\\
&\quad+\frac{i}{a!}\int_{0}^{\pi/2}\frac{2X^{2} e^{-i\theta X}}{e^{- \pi i X} - 1}\left(\frac{\pi}{2}i-\log{\tan{\frac{\theta}{2}}}\right)^{a}\,d\theta \nonumber \\
&\quad+i\frac{(i\pi/2)^{a}}{a!}X\csc\left(\frac{\pi X}{2}\right),\nonumber
\end{align}
where $B_{k}$ denotes the $k$-th Bernoulli number.
\end{theorem}
The real part of the coefficient of $X^{b}$ $(b \in \mathbb{Z}_{\ge 2})$ of the left-hand side of (\ref{eq:genT2}) is equal to $T(\{1\}^{a-1}, b)$. 
Therefore, putting $a=2n$ $(n \in \mathbb{Z}_{\ge 1})$ and taking the real part of the coefficient of $X^{b}$, we obtain
\begin{align*}
T(\{1\}^{2n-1}, b) = \Re\left(-\frac{2}{\pi(2n)!}\frac{(-\pi i)^{b-1}}{(b-1)!}\int_{0}^{\pi/2}B_{b-1}\left(\frac{\theta}{\pi}\right)\left(\frac{\pi}{2}i-\log{\tan{\frac{\theta}{2}}}\right)^{2n}\,d\theta\right),
\end{align*}
where $B_{k}(t)$ denotes the $k$-th Bernoulli polynomial defined by the generating function
\begin{align*}
\frac{Xe^{tX}}{e^{X}-1} = \sum_{k=0}^{\infty} B_{k}(t)\frac{X^{k}}{k!}.
\end{align*}
Therefore, for $k \in \mathbb{Z}_{\ge 1}$, we obtain the following corollary.
\begin{corollary}\label{co:T(2n+1,k)} For $n \in \mathbb{Z}_{\ge 1}$ and $k \in \mathbb{Z}_{\ge 1}$, we have
\begin{align*}
&T(\{1\}^{2n-1}, 2k) \\
&= \frac{\pi^{2k-1}}{(2k-1)!}\sum_{m=0}^{n-1}\frac{(-1)^{k+m-1}(\pi/2)^{2m}}{(2m+1)!(2n-2m-1)!}\int_{0}^{\pi/2}B_{2k-1}\left(\frac{\theta}{\pi}\right)\left(\log{\tan{\frac{\theta}{2}}}\right)^{2n-2m-1}\,d\theta
\end{align*}
and
\begin{align*}
&T(\{1\}^{2n-1}, 2k+1) \\
&=\frac{\pi^{2k}}{(2k)!}\sum_{m=0}^{n-1}\frac{(-1)^{k+m-1}(\pi/2)^{2m-1}}{(2m)!(2n-2m)!}\int_{0}^{\pi/2}B_{2k}\left(\frac{\theta}{\pi}\right)\left(\log{\tan{\frac{\theta}{2}}}\right)^{2n-2m}\,d\theta.
\end{align*}
\end{corollary}
The second equation follows from $\int_{0}^{\pi/2}B_{2k}\left(\theta/\pi\right)\,d\theta = 0$. In particular, when $n=1$, we have
\begin{align}
T(1, 2k) = \frac{(-1)^{k-1}\pi^{2k-1}}{(2k-1)!}\int_{0}^{\pi/2}B_{2k-1}\left(\frac{\theta}{\pi}\right)\log{\tan{\frac{\theta}{2}}}\,d\theta \label{eq:T(1,2k)=}
\end{align}
and
\begin{align*}
T(1, 2k+1)
&=\frac{(-1)^{k-1}\pi^{2k-1}}{(2k)!}\int_{0}^{\pi/2}B_{2k}\left(\frac{\theta}{\pi}\right)\left(\log{\tan{\frac{\theta}{2}}}\right)^{2}\,d\theta.
\end{align*}
In \cite[equation (2.3)]{EG2}, it was shown that
\begin{align}
\int_{0}^{\pi/2}B_{2m-1}\left(\frac{2}{\pi}x\right)\log{\tan{x}}\,dx = \frac{2(-1)^{m-1}(2m-1)!}{(2\pi)^{2m-1}}\sum_{n=1}^{\infty}\frac{h_{n}}{n^{2m}}, \label{eq:T(1,2k)EG}
\end{align}
where $h_{n} = \sum_{k=1}^{n}(2k-1)^{-1}$. By substituting $x = \theta/2$ and using the symmetry $B_{2m-1}(1-x)\log{\tan(\pi/2-x)} = B_{2m-1}(x)\log{\tan{x}}$, we can see that the equation (\ref{eq:T(1,2k)=}) and the equation (\ref{eq:T(1,2k)EG}) are equivalent.
\begin{proof}[Proof of Theorem \ref{th:genT2}]
By the formula (\ref{eq:Tnk}), we have
\begin{align*}
&\sum_{k=1}^{b}\frac{B_{b-k}}{(b-k)!}\left(-\pi i \right)^{b-k-1}T(\{1\}^{a-1},k+1)\\
&=\frac{i^{a}}{(a-1)!}\sum_{k=0}^{b}\frac{B_{b-k}}{k!(b-k)!}\left(-\pi i \right)^{b-k-1}\int_{0}^{\pi/2}\theta^{a-1}\left(-\frac{\pi}{2}i-\log{\tan{\frac{\theta}{2}}}\right)^{k}\,d\theta\\
&\quad+\frac{1}{a!}\sum_{k=1}^{b}\frac{(-i)^{k}B_{b-k}}{(k-1)!(b-k)!}\left(- \pi i \right)^{b-k-1}\int_{0}^{\pi/2}\theta^{k-1}\left(\frac{\pi}{2}i-\log{\tan{\frac{\theta}{2}}}\right)^{a}\,d\theta\\
&\quad-\frac{(i\pi/2)^{a}}{a!}\sum_{k=0}^{b}\frac{B_{b-k}}{(b-k)!}\left(-\pi i \right)^{b-k-1}\frac{(-i\pi/2)^{k}}{k!}.
\end{align*}
Since
\begin{align*}
\sum_{k=0}^{\infty}\frac{B_{k}}{k!}\left(- \pi i \right)^{k} X^{k} = \frac{\pi X}{2}\cot\left(\frac{\pi X}{2}\right)+\frac{\pi X}{2}i = \frac{-\pi i X}{e^{- \pi i X} - 1},
\end{align*}
it follows that
\begin{align*}
&\sum_{b=1}^{\infty}\sum_{k=1}^{b}\frac{B_{b-k}}{(b-k)!}\left(- \pi i \right)^{b-k-1}T(\{1\}^{a-1},k+1) X^b\\
&=\frac{i^{a}}{(a-1)!}\frac{ X}{e^{-\pi i X} - 1}\int_{0}^{\pi/2}\theta^{a-1}e^{\left(-i\pi/2-\log{\tan(\theta/2)}\right)X}\,d\theta\\
&\quad-\frac{i}{a!}\frac{X}{e^{-\pi i X} - 1}\int_{0}^{\pi/2}Xe^{-i\theta X}\left(\frac{\pi}{2}i-\log{\tan{\frac{\theta}{2}}}\right)^{a}\,d\theta\\
&\quad-\frac{(i\pi/2)^{a}}{a!}\frac{X}{e^{- \pi i X} - 1}e^{(-i\pi /2)X}.
\end{align*}
Therefore, the desired formula is obtained.
\end{proof}

\section{relation between multiple $A$-values and iterated log-tangent integrals}\label{sec:AandLt}
In this section, we evaluate iterated log-tangent integrals in terms of multiple $T$-values and multiple $A$-functions and prove Theorem \ref{th:L=A}.
\subsection{Evaluation iterated log-tangent integrals}
This subsection proves that every iterated log-tangent integral can be written in terms of multiple $T$-values and multiple $A$-functions. More precisely, we prove the following theorem.
\begin{theorem}\label{th:LttoF}
For $\sigma \in [0,\pi]$, $\mathbf{k} \in \mathbb{Z}_{\ge 0}^{n}$ and  $\mathbf{l}  \in \mathbb{Z}_{\ge 0}^{n}$ such that $\mathbf{k} - \mathbf{1}_{n} \ge \mathbf{l}$, we have
\begin{align*}
\Lt_{\mathbf{k}}^{\mathbf{l}}(\sigma) = (-1)^{|\mathbf{k}|+n}i^{|\mathbf{l}|+n}\sum_{\mathbf{p}+\mathbf{r}=\mathbf{k}-\mathbf{1}_{n}-\mathbf{l}}\binom{\mathbf{k}-\mathbf{1}_{n}-\mathbf{l}}{\mathbf{p},\mathbf{r}}\left(-\frac{\pi }{2}i\right)^{|\mathbf{p}|}\tilde{F}_{\mathbf{l}}^{\mathbf{r}}(\sigma),
\end{align*}
where the sum is over all $\mathbf{p}\in \mathbb{Z}_{\ge0}^{n}$ and $\mathbf{r}\in \mathbb{Z}_{\ge0}^{n}$ satisfying $\mathbf{p}+\mathbf{r}=\mathbf{k}-\mathbf{1}_{n}-\mathbf{l}$, and
\[\binom{\mathbf{k}-\mathbf{1}_{n}-\mathbf{l}}{\mathbf{p},\mathbf{r}} = \prod_{u=1}^{n}\frac{(k_{u}-1-l_{u})!}{p_{u}!r_{u}!}.\]
\end{theorem}
The function $\tilde{F}_{\mathbf{l}}^{\mathbf{r}}(\sigma)$ will be defined in proof of Theorem \ref{th:LttoF} and is written as a $\mathbb{Q}$-linear combination of 
\begin{align*}
\left\{ (i\sigma)^{m} T(\mathbf{k}_{1}) A(\mathbf{k}_{2};e^{i \sigma}) \mathrel{}\middle|\mathrel{}  \begin{gathered}m \in \mathbb{Z}_{\ge 0},\ \mathbf{k}_{1}, \mathbf{k}_{2}: \text{admissible},\\
m+|\mathbf{k}_{1}| + |\mathbf{k}_{2}| = |\mathbf{r}|+|\mathbf{l}|+n,\\
{\rm dep}(\mathbf{k}_{1})+{\rm dep}(\mathbf{k}_{2}) = |\mathbf{r}|
\end{gathered}\right\}.
\end{align*}
Therefore, Theorem \ref{th:LttoF} states that $\Lt_{\mathbf{k}}^{\mathbf{l}}(\sigma)$ can be written as a $\mathbb{Q}$-linear combination of
\begin{align*}
\left\{ i^{|\mathbf{l}|+n}  (i\sigma)^{m} (\pi i)^{p}T(\mathbf{k}_{1}) A(\mathbf{k}_{2};e^{i \sigma}) \mathrel{}\middle|\mathrel{}  \begin{gathered}m, p \in \mathbb{Z}_{\ge 0},\ 
\mathbf{k}_{1}, \mathbf{k}_{2}: \text{admissible},\\
m+p+|\mathbf{k}_{1}| + |\mathbf{k}_{2}| = |\mathbf{k}|,\\
p + {\rm dep}(\mathbf{k}_{1})+{\rm dep}(\mathbf{k}_{2}) = |\mathbf{k}-\mathbf{1}_{n}-\mathbf{l}|
 \end{gathered}\right\}.
\end{align*}
Theorem \ref{th:LttoF} can be used for numerical evaluations of the iterated log-tangent integral. Theorem \ref{th:LttoF} is a log-tangent analogue of \cite[Theorem 1]{U2} and is proved in the same way as that theorem. 

First, we define $A_{\mathbf{q}}^{\mathbf{r}}(\sigma)$ as follows.
\begin{definition}
For $\mathbf{q} = (q_{1},\dots,q_{n}) \in \mathbb{Z}_{\ge 0}^{n}$, $\mathbf{r} = (r_{1},\dots,r_{n}) \in \mathbb{Z}_{\ge 0}^{n}$ and $\sigma \in [0,\pi]$, we define
\begin{align*}
A_{\mathbf{q}}^{\mathbf{r}}(\sigma) &= \int_{0}^{\sigma}\int_{0}^{\theta_{n}}\cdots\int_{0}^{\theta_{2}}
\prod_{u=1}^{n}i(i\theta_{u})^{q_{u}}\left(A(1;e^{i\theta_{u}})\right)^{r_{u}}\,d\theta_{1}\cdots d\theta_{n-1}d\theta_{n}\\
&=\int_{0}^{\sigma}i(i\theta_{n})^{q_{n}}\left(A(1;e^{i\theta_{n}})\right)^{r_{n}} A_{\mathbf{q}_{-}}^{\mathbf{r}_{-}}(\theta_{n})\,d\theta_{n},
\end{align*}
where $A_{\emptyset}^{\emptyset}(\sigma)$ is regarded as $1$.
\end{definition}
Then, from $A(1;e^{i \theta}) = -\log{\tan(\theta/2)} + i \pi/2$ for $\theta \in (0, \pi)$, we have
\begin{align*}
&(-1)^{|\mathbf{k}|+n}(-i)^{|\mathbf{l}|+n}\Lt_{\mathbf{k}}^{\mathbf{l}}(\sigma)\\
&=(-1)^{|\mathbf{k}|}(-i)^{|\mathbf{l}|+n}\int_{0}^{\sigma}\int_{0}^{\theta_{n}}\dots\int_{0}^{\theta_{2}}
\prod_{u=1}^{n}\theta_u^{l_u}\left(\log\tan{\frac{\theta_{u}}{2}}\right)^{k_u-1-l_u}\,d\theta_{1}\cdots d\theta_{n}\\
&=\int_{0}^{\sigma}\int_{0}^{\theta_{n}}\dots\int_{0}^{\theta_{2}}
\prod_{u=1}^{n}i(i \theta_u)^{l_u}\left(A(1;e^{i \theta}) - \frac{\pi}{2}i\right)^{k_u-1-l_u}\,d\theta_{1}\cdots d\theta_{n}\\
&= \sum_{\mathbf{p}+\mathbf{r}=\mathbf{k}-\mathbf{1}_{n}-\mathbf{l}}\binom{\mathbf{k}-\mathbf{1}_{n}-\mathbf{l}}{\mathbf{p},\mathbf{r}}\left(-\frac{\pi }{2}i\right)^{|\mathbf{p}|}A_{\mathbf{l}}^{\mathbf{r}}(\sigma).
\end{align*}
Therefore, to prove Theorem \ref{th:LttoF}, we need only prove that $A_{\mathbf{l}}^{\mathbf{r}}(\sigma) = \tilde{F}_{\mathbf{l}}^{\mathbf{r}}(\sigma)$.
In order to define $\tilde{F}_{\mathbf{l}}^{\mathbf{r}}(\sigma)$, we introduce several notations. The following notations are used with the same meaning as in \cite{U2}: the rational numbers $B_{\mathbf{q}}$, $C_{\mathbf{q}}^{\mathbf{j}}$, the relation between indices $\preceq$, the elements $w_{\mathbf{j}}^{\mathbf{r}}$ in $\mathfrak{H}$. For a pair of indices $\mathbf{q}, \mathbf{r}\in \mathbb{Z}_{\ge0}^{n}$, we also write
\begin{align*}
\mathbf{r}&=(\mathbf{r}',\mathbf{r}'') = (\underbrace{0, \cdots, 0}_{n'},r''_{1},\dots,r''_{n''})\ \ \ \, \in \mathbb{Z}_{\ge 0}^{n'+n''},\ (r''_{1} \ge 1),\\
\mathbf{q}&=(\mathbf{q}',\mathbf{q}'')=(q'_{1},\dots,q'_{n'},q''_{1},\dots,q''_{n''})\in \mathbb{Z}_{\ge 0}^{n'+n''},\\
\overline{\mathbf{q}} &= \begin{cases} (|\mathbf{q}'|+n'+q_{1}'',q_{2}'',\dots,q_{n''}'') &\text{if}\quad \mathbf{q}''\neq\emptyset, \\ 
\emptyset &\text{if}\quad \mathbf{q}''=\emptyset.
\end{cases}
\end{align*}
The function $\tilde{f}_{\mathbf{q}}^{\mathbf{r}}(\sigma)$, which is necessary to define $\tilde{F}_{\mathbf{q}}^{\mathbf{r}}(\sigma)$, is defined as follows.
\begin{definition}\label{def:ftilde}
For $\mathbf{q} \in \mathbb{Z}_{\ge 0}^{n}$, $\mathbf{r} \in \mathbb{Z}_{\ge 0}^{n}$ and $\sigma \in [0,\pi]$, we define
\begin{align*}
\tilde{f}_{\mathbf{q}}^{\mathbf{r}}(\sigma) = 
B_{\mathbf{q}'}\sum_{\mathbf{j} \preceq \overline{\mathbf{q}}}C_{\overline{\mathbf{q}}}^{\mathbf{j}}(i\sigma)^{|\mathbf{q}|+n'-|\mathbf{j}|}A(w_{\mathbf{j}}^{\mathbf{r}''};e^{i\sigma}),
\end{align*}
where the sum is over all $\mathbf{j}=(j_{1},\dots,j_{n''}) \in \mathbb{Z}_{\ge 0}^{n''}$ satisfying $\mathbf{j} \preceq \overline{\mathbf{q}}$.
\end{definition}
The function $\tilde{f}_{\mathbf{q}}^{\mathbf{r}}(\sigma)$ satisfies the following three formulas corresponding to equations (2.1), (2.2) and Proposition 1 in \cite{U2}, respectively: 
\begin{align*}
\tilde{f}_{\mathbf{q}}^{\mathbf{r}}(\sigma) &=
\begin{cases}
1 &{\rm if}\ \mathbf{r}=\emptyset,\\
B_{\mathbf{q}}(i\sigma)^{|\mathbf{q}|+n}
&{\rm if}\ \mathbf{r}=(\{0\}^{n}), n\ge 1,\\
B_{\mathbf{q}'}\sum_{\mathbf{j} \preceq \overline{\mathbf{q}}}C_{\overline{\mathbf{q}}}^{\mathbf{j}}(i\sigma)^{|\overline{\mathbf{q}}|-|\mathbf{j}|}A(w_{\mathbf{j}}^{\mathbf{r}''};e^{i\sigma})&{\rm otherwise}.
\end{cases}\\
\tilde{f}_{\mathbf{q}}^{\mathbf{r}}(0) &= 
\begin{cases}
1 &{\rm if}\ \mathbf{r}=\emptyset,\\
0&{\rm if}\ \mathbf{r}=(\{0\}^{n}), n\ge 1,\\
\begin{gathered}
B_{\mathbf{q}'}\sum_{\mathbf{j}_{-} \preceq \overline{\mathbf{q}}_{-}}C_{\overline{\mathbf{q}}_{-}}^{\mathbf{j}_{-}}(-1)^{|\overline{\mathbf{q}}|-|\mathbf{j}_{-}|}(|\overline{\mathbf{q}}|-|\mathbf{j}_{-}|)!\\
\times T\left((w_{\mathbf{j}_{-}}^{\mathbf{r}''_{-}} \shuffle e_{1}^{\shuffle r_{n}})e_{0}^{1+|\overline{\mathbf{q}}|-|\mathbf{j}_{-}|}\right)
\end{gathered}&{\rm otherwise}.
\end{cases}\\
\frac{d}{d\sigma}\tilde{f}_{\mathbf{q}}^{\mathbf{r}}(\sigma) &= i(i\sigma)^{q_{n}}\left(A\left(1;e^{i\sigma}\right)\right)^{r_{n}}\tilde{f}_{\mathbf{q}_{-}}^{\mathbf{r}_{-}}(\sigma).
\end{align*}
The function $\tilde{F}_{\mathbf{q}}^{\mathbf{r}}(\sigma)$ is defined as follows.
\begin{definition}
For $\mathbf{q} \in \mathbb{Z}_{\ge 0}^{n}$, $\mathbf{r} \in \mathbb{Z}_{\ge 0}^{n}$ and $\sigma \in [0,\pi]$, we define
\begin{align*}\tilde{F}_{\mathbf{q}}^{\mathbf{r}}(\sigma) =\sum_{\substack{(\mathbf{q}^{(1)},\dots,\mathbf{q}^{(h)})=\mathbf{q} \\ (\mathbf{r}^{(1)},\dots,\mathbf{r}^{(h)})=\mathbf{r}}}(-1)^{h-1}\left(\prod_{j=1}^{h-1}\tilde{f}_{\mathbf{q}^{(j)}}^{\mathbf{r}^{(j)}}(0)\right)\left(\tilde{f}_{\mathbf{q}^{(h)}}^{\mathbf{r}^{(h)}}(\sigma)-\tilde{f}_{\mathbf{q}^{(h)}}^{\mathbf{r}^{(h)}}(0)\right),
\end{align*}
where the sum is over all partitions of $\mathbf{q}$ and $\mathbf{r}$ that satisfy ${\rm dep}(\mathbf{q}^{(i)})={\rm dep}(\mathbf{r}^{(i)}) \ge 1$. For example, 
\begin{align*}
\tilde{F}_{q_{1}, q_{2}, q_{3}}^{(r_{1}, r_{2}, r_{3})}(\sigma) &= (\tilde{f}_{q_{1}, q_{2}, q_{3}}^{(r_{1}, r_{2}, r_{3})}(\sigma)-\tilde{f}_{q_{1}, q_{2}, q_{3}}^{(r_{1}, r_{2}, r_{3})}(0))\\
&\quad-\tilde{f}_{q_{1}}^{(r_{1})}(0)(\tilde{f}_{q_{2}, q_{3}}^{(r_{2}, r_{3})}(\sigma)-\tilde{f}_{q_{2}, q_{3}}^{(r_{2}, r_{3})}(0)) - \tilde{f}_{q_{1}, q_{2}}^{(r_{1}, r_{2})}(0)(\tilde{f}_{q_{3}}^{(r_{3})}(\sigma)-\tilde{f}_{q_{3}}^{(r_{3})}(0))\\
&\quad + \tilde{f}_{q_{1}}^{(r_{1})}(0)\tilde{f}_{q_{2}}^{(r_{2})}(0)(\tilde{f}_{q_{3}}^{(r_{3})}(\sigma)-\tilde{f}_{q_{3}}^{(r_{3})}(0)).
\end{align*}
\end{definition}
The function $\tilde{F}_{\mathbf{q}}^{\mathbf{r}}(\sigma)$ satisfy 
\begin{align*}
\frac{d}{d\sigma}\tilde{F}_{\mathbf{q}}^{\mathbf{r}}(\sigma) = i(i\sigma)^{q_{n}}\left(A\left(1;e^{i\sigma}\right)\right)^{r_{n}}\tilde{F}_{\mathbf{q}_{-}}^{\mathbf{r}_{-}}(\sigma)
\end{align*}
corresponding to \cite[Proposition 2]{U2}, where $\tilde{F}_{\emptyset}^{\emptyset}(\sigma)$ is regarded as $1$. Therefore, we can verify $A_{\mathbf{l}}^{\mathbf{r}}(\sigma) = \tilde{F}_{\mathbf{l}}^{\mathbf{r}}(\sigma)$ as in \cite[Proposition 3]{U2}, which completes the proof of Theorem \ref{th:LttoF}.
\subsection{Examples of Theorem \ref{th:LttoF}}
Here are some examples of Theorem \ref{th:LttoF}. Note that when $\mathbf{k} = (k)$ and $\mathbf{l} = (k-2)$, Theorem \ref{th:LttoF} becomes the simple formula:
\begin{align*}
\Lt_{k}^{(k-2)}(\sigma) &= -\frac{\pi i}{2(k-1)}\sigma^{k-1} + i^{k-1}(k-2)!T(k)\\
&\quad-\sum_{j=0}^{k-2}i^{j+1}\frac{(k-2)!}{(k-2-j)!}\sigma^{k-2-j}A(j+2;e^{i\sigma}).
\end{align*}
The following equations are examples of Theorem \ref{th:LttoF}:
\begin{align*}
&\Lt_{1}^{(0)}(\sigma) = -\sigma,\\
&\Lt_{2}^{(0)}(\sigma) = -\frac{1}{2} i \, \pi \sigma + i \, T\left(2\right) - i \, A\left(2;e^{i \, \sigma}\right),\\
&\Lt_{2}^{(1)}(\sigma) = -\frac{1}{2} \, \sigma^{2},\\
&\Lt_{1,1}^{(0,0)}(\sigma) = \frac{1}{2} \, \sigma^{2},\\
&\Lt_{3}^{(0)}(\sigma) = \frac{1}{4} \, \pi^{2} \sigma - \pi T\left(2\right) + \pi {A}\left(2;e^{i \, \sigma}\right) - 2 i \, T\left(1 , 2\right) + 2 i \, A\left(1,2;e^{i \, \sigma}\right),\\
&\Lt_{3}^{(1)}(\sigma) = -\frac{1}{4} i \, \pi \sigma^{2} - i \, \sigma A\left(2;e^{i \, \sigma}\right) - T\left(3\right) + A\left(3;e^{i \, \sigma}\right),\\
&\Lt_{3}^{(2)}(\sigma) = -\frac{1}{3} \, \sigma^{3},\\
&\Lt_{1,2}^{(0,0)}(\sigma) = \frac{1}{4} i \, \pi \sigma^{2} + i \, \sigma A\left(2;e^{i \, \sigma}\right) + T\left(3\right) - A\left(3;e^{i \, \sigma}\right),\\
&\Lt_{1,2}^{(0,1)}(\sigma) = \frac{1}{3} \, \sigma^{3},\\
&\Lt_{2,1}^{(0,0)}(\sigma) = \frac{1}{4} i \, \pi \sigma^{2} - i \, \sigma T\left(2\right) - T\left(3\right) + A\left(3;e^{i \, \sigma}\right),\\
&\Lt_{2,1}^{(1,0)}(\sigma) = \frac{1}{6} \, \sigma^{3},\\
&\Lt_{1,1,1}^{(0,0,0)}(\sigma) = -\frac{1}{6} \, \sigma^{3}.
\end{align*}
\subsection{Proof of Theorem \ref{th:L=A}} As a consequence of Theorem \ref{th:LttoF}, we give a proof of Theorem \ref{th:L=A}.
The inclusions $\mathcal{A}_{k}^{e} \subset \mathcal{L}^{e}_{k}$ and $\mathcal{A}_{k}^{o} \subset \mathcal{L}^{o}_{k}$ have already been proved in (\ref{eq:AinL}). By Theorem \ref{th:LttoF}, the iterated log-tangent integral $\Lt_{\mathbf{k}}^{\mathbf{l}}(\pi/2)$ can be written as a $\mathbb{Q}$-linear combination of
\begin{align*}
\left\{ i^{|\mathbf{l}|+n} (\pi i)^{m+p} T(\mathbf{k}_{1}) A(\mathbf{k}_{2};i) \mathrel{}\middle|\mathrel{}  \begin{gathered}m, p \in \mathbb{Z}_{\ge 0},\ 
\mathbf{k}_{1}, \mathbf{k}_{2}: \text{admissible},\\
m+p+|\mathbf{k}_{1}| + |\mathbf{k}_{2}| = |\mathbf{k}|,\\
p + {\rm dep}(\mathbf{k}_{1})+{\rm dep}(\mathbf{k}_{2}) = |\mathbf{k}-\mathbf{1}_{n}-\mathbf{l}|
 \end{gathered}\right\}.
\end{align*}
Since $(\pi i)^{m+p} = 2^{m+p}A(1;i)^{m+p}$ and Corollary \ref{co:TtoA}, it follows that 
$(\pi i)^{m+p} T(\mathbf{k}_{1}) A(\mathbf{k}_{2};i)$ can be written as a $\mathbb{Q}$-linear combination of multiple $A$-values of weight $|\mathbf{k}|$ by using the product formula for multiple $A$-functions. In other words, $\Lt_{\mathbf{k}}^{\mathbf{l}}(\pi/2)$ can be written as a $\mathbb{Q}$-linear combination of $i^{|\mathbf{l}|+n}A(\mathbf{k}'; i)$ where  $i^{|\mathbf{l}|+n}A(\mathbf{k}'; i)$ is a real number and satisfies $|\mathbf{k}'| = |\mathbf{k}|$. From the definitions of $A^{e}$ and $A^{o}$, it follows that
\begin{align*}
A(\mathbf{k}'; i)&=\begin{cases}
A^{e}(\mathbf{k}';\pi/2) + iA^{o}(\mathbf{k}';\pi/2)\quad \text{if}\ |\mathbf{k}'|:{\rm even},\\
A^{o}(\mathbf{k}';\pi/2) + iA^{e}(\mathbf{k}';\pi/2)\quad \text{if}\ |\mathbf{k}'|:{\rm odd}.
\end{cases}
\end{align*}
Therefore, we have
\begin{align*}
i^{|\mathbf{l}|+n}A(\mathbf{k}'; i)&=\begin{cases}
\mathcal{A}_{|\mathbf{k}'|}^{e} + i\mathcal{A}_{|\mathbf{k}'|}^{o}\quad \text{if}\ |\mathbf{k}'|+|\mathbf{1}_{n}+\mathbf{l}|:{\rm even},\\
\mathcal{A}_{|\mathbf{k}'|}^{o} + i\mathcal{A}_{|\mathbf{k}'|}^{e}\quad \text{if}\ |\mathbf{k}'|+|\mathbf{1}_{n}+\mathbf{l}|:{\rm odd}.
\end{cases}
\end{align*}
By the condition that $i^{|\mathbf{l}|+n}A(\mathbf{k}'; i)$ is a real number and $|\mathbf{k}'| = |\mathbf{k}|$, we have
\begin{align*}
i^{|\mathbf{l}|+n}A(\mathbf{k}'; i) \in \begin{cases} 
\mathcal{A}_{|\mathbf{k}|}^{e}&{\rm if}\quad |\mathbf{k}-\mathbf{1}_{n}-\mathbf{l}|:{\rm even},\\ 
\mathcal{A}_{|\mathbf{k}|}^{o}&{\rm if}\quad |\mathbf{k}-\mathbf{1}_{n}-\mathbf{l}|:{\rm odd},
\end{cases}
\end{align*}
which completes the proof of Theorem \ref{th:L=A}.
\section*{Acknowledgment}
The author is deeply grateful to Prof. Masanobu Kaneko for their helpful comments.

\vspace{4mm}

{\footnotesize
{\sc
\noindent
Graduate School of Mathematics, Nagoya University,\\
Chikusa-ku, Nagoya 464-8602, Japan.
}\\
{\it E-mail address}, R. Umezawa\hspace{1.75mm}: {\tt
m15016w@math.nagoya-u.ac.jp}\\
}

\end{document}